\documentclass{amsart}
\usepackage{amssymb,latexsym}
\theoremstyle{plain}
\newtheorem{theorem}{Theorem}

\newtheorem{proposition}{Proposition}
\newtheorem{lemma}{Lemma}
\theoremstyle{definition}

\newtheorem{example}{Example}

\newtheorem{remark}{Remark}
\newtheorem{question}{Question}

\newcommand{\enm}[1]{\ensuremath{#1}}          %

\newcommand{\cal}[1]{\mathcal{#1}}

\newcommand{\NN}{\enm{\mathbb{N}}}

\newcommand{\ZZ}{\enm{\mathbb{Z}}}

\newcommand{\PP}{\enm{\mathbb{P}}}

\newcommand{\Ii}{\enm{\cal{I}}}

\newcommand{\Ll}{\enm{\cal{L}}}

\newcommand{\Oo}{\enm{\cal{O}}}

\renewcommand{\phi}{\varphi}
\renewcommand{\theta}{\vartheta}
\renewcommand{\epsilon}{\varepsilon}

\renewcommand{\to}[1][]{\xrightarrow{\ #1\ }}

\newcommand{\old}[1]{}

\date{}

\begin{document}

\title[curves]
{On the Hilbert function of intersections of a hypersurface with general reducible curves}
\author{Edoardo Ballico}
\address{Dept. of Mathematics\\
 University of Trento\\
38123 Povo (TN), Italy}
\email{ballico@science.unitn.it}
\thanks{The author was partially supported by MIUR and GNSAGA of INdAM (Italy).}

\begin{abstract}
Let $W\subset \PP^n$, $n\ge 3$, be a degree $k$ hypersurface. Consider a ``~general ~'' reducible, but connected,  curve
$Y\subset \PP^n$, for instance a sufficiently general connected and nodal union of lines with  $p_a(Y)=0$, i.e. a tree of lines. We study the Hilbert function of the
set $Y\cap W$ with cardinality $k\deg (Y)$ and prove when it is the expected one. We give complete classification of the
exceptions for $k=2$ and for $n=k=3$. We apply these results and tools to the case in which $Y$ is a smooth curve with $\Oo _Y(1)$ non-special.
\end{abstract}

\maketitle

\section{Introduction}

A degree $d$ \emph{tree} $T\subset \PP^n$, $n\ge 3$, is a connected nodal curve of degree $d$ with arithmetic genus $0$ whose irreducible components are
lines. A \emph{forest} in $\PP^r$ is a union of finitely many disjoint trees. For all positive integers $s$, $d$ and $d_i$,
$1\le i\le s$, let $T(n,d)$ denote the set of all degree $d$ trees $T\subset \PP^n$ and $T(n;s;d_1,\dots ,d_s)$ the set of all
forests in $\PP^n$ with $s$ connected components of degree $d_1,\dots ,d_s$. The sets $T(n,d)$ and $T(n;s,d_1,\dots,d_s$ are
smooth quasi-projective varieties. If $d\ge 4$ the set $T(n,d)$ is not irreducible, but we may describe its irreducible
components, which are also its connected components, in the following way. Fix an integer $d\ge 2$ and let $T\subset \PP^n$ be a
degree $d$ tree. It is easy to see the existence of an ordering $L_1,\dots ,L_d$ of the irreducible components of $T$ such that
for all $i=1,\dots ,d$ the curve $\cup _{1\le j\le i} L_j$ is connected. We will say that any such ordering is \emph{admissible}. 
Let $T(n,d)''$ be the set of pairs $(T,\leq)$, where $T\in T(n,d)$ and $\leq $ is an admissible ordering of $T$. Fix
one such ordering. Since
$T$ is nodal, for each
$i\in \{2,\dots,d\}$ there is a unique $\tau(i)\in \{1,\dots ,i-1\}$ such that $L_i\cap L_{\tau(i)}\ne \emptyset$. Thus the
admissible ordering of $T$ induces a function $\tau: \{2,\dots ,d\}\to \{1,\dots ,d-1\}$ such that $\tau(i)<i$ for all $i$.
As in \cite{b,be6} we say that $\tau$ is the \emph{type} of $(T,\leq)$. Let $\tau: \{2,\dots ,d\}\to \{1,\dots ,d-1\}$ be any map such
that $\tau(i)<i$ for all $i$. It is obvious how to construct a pair $(T,\leq)\in T(n,d)''$ such that $\tau$ is the type
of $(T,\leq)$. For any map $\tau: \{2,\dots ,d\}\to \{1,\dots ,d-1\}$ such that $\tau(i)<i$ for all $i$ let
$T(n,d,\tau)$ denote the set of all $T\in T(n,d)$ with an admissible ordering $\leq$ such that $(T,\leq)$ has type $\tau$. It
is easy to see that the sets $T(n,d,\tau)$ are the irreducible components of $T(n,d)$, except that we may have
$T(n,d,\tau)=T(n,d,\tau ')$ for some $\tau '\ne \tau$. Note that either $T(n,d,\tau)=T(n,d,\tau ')$ or
$T(n,d,\tau)\cap T(n,d,\tau ') =\emptyset$. Thus the sets $T(r,d,\tau)$ are the connected components of $T(r,d)$, too.
A degree $d$ tree $T\subset \PP^r$ is called a \emph{bamboo} if either $d=1$ or $d\ge 2$ and there is an admissible ordering
$\leq$ of $T$ with type $\tau(i)=i-1$ for all $i$. Thus a degree $d\ge 2$ tree is a bamboo if and only if there is an ordering
$L_1,\dots ,L_d$ of its irreducible components such that $L_i\cap L_j\ne \emptyset$ if and only if $|i-j|\le 1$. If
$d\le 3$ every degree $d$ tree is a bamboo. Let $B(n,d)$ denote the set of all degree $d$ bamboos contained in $\PP^n$.
For all integers $n\ge 3$, $s>0$ and $d_i>0$, $1\le i \le s$, let $B(n;s;d_1,\dots ,d_s)$ denote the subsets of
$T(n;s,d_1,\dots ,d_s)$ formed by all forests whose connected components are bamboos.
If $d=1$ we will say that $T$ is the \emph{final line} of $T$. If $d\ge 2$, a line $L\subset T$ is said to be a \emph{final line} if it meets only another irreducible component of $T$. Note that a degree $d\ge 2$ tree $T$ has at least $2$ final lines and that $T$ is a bamboo if and only if it has exactly $2$ final lines.

 Let $M\subseteq \PP^n$ be an irreducible variety. Let $E\subset M$ be any closed subscheme. We will say that $E$ has
\emph{maximal rank in $M$} if for each $t\in \NN$ the restriction map $H^0(\Oo_M(t))\to H^0(\Oo_E(t))$ has
maximal rank, i.e. it is injective or surjective. If $E$ is a finite set, $E$ has maximal rank in $M$ if and only if
$h^0(M,\Ii _{E,M}(t)) = \max \{0,h^0(\Oo_M(t))-\#E\}$ for all $t$. Fix $\Ll \in \mathrm{Pic}(M)$. We will say that $E$ has
\emph{maximal rank with respect to $\Ll$} if the restriction map $H^0(M,\Ll)\to H^0(E,\Ll _{|E})$ has maximal rank.

Why trees~? They are useful even if one is only interested in the study of smooth curves in projective spaces
(\cite{aly,b,be01,be00,be4,be5,be8,be9,hh3,hi,l8,l9}).  Why a few mathematicians care about the Hilbert function of the
intersection of a curve $C\subset \PP^n$ with a hypersurface of $\PP^n$, often a quadric for $n=3$ and a hyperplane for $n>3$
(see the string of papers just quoted and \cite{l1,l4,l5,pe,v})~? Because it is often a key step to prove that $C$ has the
expected postulation. We also point out that proving that the general element of a tiny family of curves has nice intersection
with a hypersurface may be very efficiently use for results on smooth curves (\cite{b,be6}).

We prove the following results.

\begin{theorem}\label{be2}
Let $W\subset \PP^3$ be a general cubic surface. Fix a positive integer $d$. There is a type
$\tau$ for degree
$d$ trees such that for a general $X\in T(3,d,\tau)$ the set
$X\cap W$ has maximal rank in
$W$ for the line bundle $\Oo_W(t)$ except in the following cases:
$$(t,d)\in \{(1,2),(2,3),(2,4)\}.$$
\end{theorem}

Remark \ref{eeb2} explains the exceptional cases and computes their cohomology groups. These are exceptional cases for all
degree $d$ trees transversal to $W$. In the proof of Theorem \ref{be2} we could get several different types.

\begin{question}\label{qbe2}
Is Theorem \ref{be2} true for a general $X\in B(3,d)$?
\end{question}

\begin{theorem}\label{eb1}
Fix integers $d \ge n\ge 3$. Let $W\subset \PP^n$ be a quadric hypersurface of rank $\rho \ge 4$. There is a
degree
$d$ bamboo
$T\subset
\PP^n$ intersecting transversally $W$ and such that the set $T\cap W$ has maximal rank.
\end{theorem}

For higher degree hypersurfaces we are only able to prove the following results.

\begin{proposition}\label{fe1}
Fix positive integers $n\ge 3$, $d$ and $k\ge 3$. Let $\tau$ be a type for degree $d$ trees. Let $W\subset \PP^n$ be a general
degree $k$ hypersurface and $Y$ a general element of $T(n,d,\tau)$. 
\begin{enumerate}
\item If $d \le \binom{n+t-k}{n-1}$, then $h^1(W,\Ii _{W\cap Y,W}(t))=0$.
\item If $d\ge \binom{n+t+k-2}{n-1}$, then $h^0(W,\Ii _{W\cap Y,W}(t))=0$.
\end{enumerate}
\end{proposition}

\begin{proposition}\label{fe2}
Fix positive integers $n\ge 3$, $g\ge 0$ and $k\ge 3$. Let $X$ be a smooth projective curve of genus $g$. There is an integer
$d_0(g,n)$ such that for every $d\ge d_0(g,n)$ the general degree $d$ embedding $Y\subset \PP^n$ has the following property.
 Let
$W\subset
\PP^n$ be a general degree $k$ hypersurface. 
\begin{enumerate}
\item If $d \le \binom{n+t-k}{n-1}$, then $h^1(W,\Ii _{W\cap Y,W}(t))=0$.
\item If $d\ge \binom{n+t+k-2}{n-1}$, then $h^0(W,\Ii _{W\cap Y,W}(t))=0$.
\end{enumerate}
\end{proposition}

\section{Preliminary observations}

We say that $\emptyset$ is the type of a degree $1$ tree. For all integers $n\ge 3$, positive integers $d_i$, $1\le i\le s$,
and types $\tau _i$ for degree $d_i$ trees, $1\le i\le s$ let $T(n;s;d_1,\tau_1,\dots ,d_s,\tau _s)$ denote the set of all
$T_1\cup \cdots \cup T_s\in T(n;s,d_1,\dots ,d_s)$ with $T_i$ of type $\tau_i$. $T(n;s,d_1,\tau_1,\dots ,d_s,\tau _s)$ is an
irreducible and smooth quasi-projective variety. 

\begin{remark}\label{t1}
Fix a hyperplane $H\subset \PP^n$, $n\ge 3$, a positive integer $d$, a type $\tau: \{2,\dots ,d\}\to \{1,\dots ,d-1\}$ for
degree $d$ trees and a finite set
$S\subset H$ such that
$\#S =d$. Since any two points of $\PP^n$ are collinear, it is easy to show the existence of $T\in T(n,d,\tau)$ such that
$T\cap H=S$. Note that any such $T$ intersects transversally $H$.
\end{remark}

\begin{remark}\label{t2}
Fix a hyperplane $H\subset \PP^n$, $n\ge 3$, a positive integer $d$, a type $\tau: \{2,\dots ,d\}\to \{1,\dots ,d-1\}$ for
degree $d$ trees and a finite set
$A\subset H$ such that $\#A\le \lfloor d/2\rfloor$. Since any two points of $\PP^n$ are collinear, it is easy to show the existence of $T\in
B(n,d)$ such that
$A = \mathrm{Sing}(T)$ and no irreducible component of $T$ is contained in $H$.
\end{remark}

Remark \ref{t2} is not true for many types $\tau$ of degree $d\ge 4$ trees. We just give an example.
\begin{example}
Consider the function $\tau: \{2,\dots ,d\}\to \{1,\dots ,d-1\}$ with $\tau(i)=1$ for all $i$. We call trees with this type
\emph{spreading} trees. Assume $d\ge 4$. Fix a hyperplane $H\subset \PP^n$ and let $T=L_1\cup \dots \cup L_d$ be a spreading
tree. Since $\mathrm{Sing}(T)\subset L_1$, any hyperplane containing $2$ singular points of $T$ contains $L_1$.
\end{example}

\begin{remark}
Take a general $T =T_1\cup \cdots \cup T_s\in T(n;s,d_1,\dots ,d_s)$ with $T_i$ of type $\tau_i$. In all cases considered in
this paper,
$T\cap W$ has maximal rank if all $T_i\cap W$ have maximal rank.
\end{remark}

We recall the following result, proved in a preliminary version of \cite{be01}; the case in which all types are bamboos is
stated in \cite[Claim at p. 592]{be6}.

\begin{proposition}\label{be1}
Fix positive integers $a$, $b$, $d$. Take a type $\tau$ for degree $d$ trees of $\PP^3$
and a general $T\in T(3,d,\tau)$. The set $T\cap Q$ with cardinality $2d$ has maximal rank with respect to the line bundle
$\Oo_Q(a,b)$, unless $(a,b,d) =(1,1,2)$.
\end{proposition}

\begin{lemma}\label{fe3}
Fix integers $n\ge 3$ and $g\ge 0$. Let $X$ be a smooth curve of genus $g$. Fix a hyperplane $H\subset \PP^n$. There is an
integer $d_0(g,n)$ for that $H\cap Y$ has maximal rank in $H$ for a general degree $d$ embedding $Y\subset \PP^n$ of $X$.
\end{lemma}

\begin{proof}
First consider a general non-special embedding $E\subset \PP^n$ of $X$ with $\deg (E)=g+n$. Since $E$ is general, $E$ is
transversal to $H$. Let $z$ be the first positive integer  such that $h^1(H,\Ii_{H\cap E}(z))=0$. The
Castelnuovo-Mumford's lemma gives $h^1(H,\Ii_{H\cap E}(x))=0$ for all $x\ge z$. Let
$S\subset H$ be a general finite subset. We have $h^1(H,\Ii _{(H\cap E)\cup S}(t)) = 0$ for all integers $t\ge k$ such that
$g+n+\#S \le \binom{n+t-1}{n}$ and $h^1(H,\Ii _{(H\cap E)\cup S}(t)) = 0$ for all integers $t\ge z$ such that $g+n+\#S \ge
\binom{n+t-1}{n}$. Thus $(E\cap E)\cup S$ has maximal rank in $H$ if $\#S\ge \binom{n+z-1}{n-1}-g-n$. There is a bamboo
$T\subset \PP^n$ such that $T\cap H =S$ and $T$ meets quasi-transversally $E$ at a unique point, $p$, with $p\notin H$.
The curve $E\cup T$ is a flat limit of a family of degree $d$ embedding of $X$ (\cite{be8,be9}).
\end{proof}

\begin{proof}[Proof of Proposition \ref{fe2}:]
Let $H$ be a hyperplane. As in the proof of Proposition \ref{fe1} it is sufficient to prove that $H\cap Y$ has maximal rank in
$H$ for a general degree $d$ embedding $Y$ of $X$. This is true by Lemma \ref{fe3}.
\end{proof}

\begin{remark}\label{fe4}
The proof of Lemma \ref{fe2} shows how to get a not very large $d_0(g,n)$. Take a general degree $g+n$ embedding $E\subset
\PP^n$ of $X$. It is sufficient to find the first positive integer $k$ such that
$h^1(H,\Ii_{E\cap H}(k)) =0$. In arbitrary characteristic the generality of $E$ shows that $E\cap H$ is in linear general
position in $H$ (\cite[Corollaries 3.5 and 3.6]{he}). Thus we may take $k=\lceil (g+n-1)/n\rceil$. Thus we may take $k=2$ for
all $g\le n+1$. When $k=2$ we may say more, because obviously $E\cap H$ spans $H$. Thus in all cases with $k=2$ we may take
$d_0(g,n)=g+n$.
\end{remark}

\section{Proof of Theorem \ref{eb1}}

Let $H\subset \PP^n$ , $n\ge 4$,be a general
hyperplane. Set $D:=  W\cap H$. For any integer $t\ge 0$ we have $h^0(\Oo_W(t)) =\binom{n+t}{n} -\binom{n+t-2}{n}$.  Thus for
all integers $t>0$ we have $h^0(\Oo_W(t)) -h^0(\Oo_W(t-1)) = h^0(\Oo _D(t)) =
\binom{n+t-1}{n-1} -\binom{n+t-3}{n-1}$.

\begin{remark}\label{eb3}
Take  $t>0$ and assume the existence of a degree $\lceil h^0(\Oo _W(t))/2\rceil$ tree $T\subset \PP^n$ 
transversal to $W$ with $h^0(W,\Ii_{T\cap W}(t)) =0$. If $h^0(\Oo_W(t))$ is even, then
$h^1(W,\Ii_{T\cap W}(t)) =0$. Now assume $h^0(\Oo _W(t))$ odd. Thus the set $S:= T\cap W$ has cardinality $h^0(\Oo _W(t))+1$
and
$h^0(W,\Ii_S(t)) =0$, there is $S'\subset S$ such that $\#S' =\#S-1$ and $h^0(W,\Ii_{S'}(t)) =0$. Thus $h^1(W,\Ii_{S'}(t))
=0$. Set $\{o\}:= S\setminus S'$ and call $L_o$ the irreducible component of $T$ containing $o$. The point $o\in T\cap W$ is
called a \emph{linking point} and $L_o$ a \emph{linking line} or a \emph{linking component} of $T$. Since $S'$ is not unique, in general, $T$ may have
several linking points and several linking lines. A linking component contains at most $2$ linking points. For a general
$T$ of prescribed type a linking components has two linking points.  Since a linking
component is a final component, a bamboo as at most $2$ linking components. A general bamboo of degree $\lceil h^0(\Oo
_W(t))/2\rceil$ has at least one linking component if and only if $R(n,t)$ is true. If $R(n,t)$ is true a general bamboo of degree $\lceil h^0(\Oo
_W(t))/2\rceil$ has $2$ linking components.
\end{remark}

Note that for any tree $T$ intersecting transversally $W$ the integer $\#T\cap W$ is even. For all integers $n\ge 4$ and $t\ge
2$ we define the following assertion
$R(n,t)$ whose formulation depends on the parity class of $h^0(\Oo_W(t))$.

\quad {\bf Assertion} $R(n,t)$: If $h^0(\Oo_W(t))$ is even, assume the existence of a degree $h^0(\Oo_W(t))/2$ bamboo
$T\subset \PP^n$ intersecting transversally $W$ and with $h^0(W,\Ii_{W\cap T}(t))=0$. If
$h^0(\Oo_W(t))$ is odd
$R(n,t)$ requires the existence of a bamboo
$T\subset
\PP^n$ such that
$\deg (T) =
\lceil h^0(\Oo_W(t))/2\rceil$, $T$ intersects transversally $W$, $h^0(W,\Ii_{T\cap W}(t)) =0$ and $T$ has a linking line which
is a final line.

\begin{remark}\label{eb2}
Take $n\ge 4$ and $t>0$ and assume the existence of a degree $x:=\lceil h^0(\Oo _W(t))/2\rceil$ tree $T\subset \PP^n$ satisfying $R(n,t)$,
i.e. assume that $T$ is transversal to $W$ and $h^0(W,\Ii_{T\cap W}(t)) =0$. Fix an admissible ordering $L_1,\dots ,L_x$ of the
irreducible components of $T$ and call $\tau$ the associated type of $T$

\quad (a) Assume that $h^0(\Oo _W(t))$ is even, i.e. assume $h^0(\Oo _W(t))=2x$. First take $d<t$. For any integer $d$ with $1\le d <x$ set $T_d:= L_1\cup \cdots \cup L_x$. The tree $T_d$
satisfy $h^1(W,\Ii _{W\cap T_d}(t))=0$. Now assume $d\ge x$ and call $T_d$ any degree $d$ tree containing $T$. Since $h^1(W,\Ii_{T\cap W}(t)) =0$, $h^1(W,\Ii_{T_d\cap W}(t)) =0$.

\quad (b) Assume $h^0(\Oo _W(t))$ odd, i.e. assume $h^0(\Oo _W(t))=2x-1$. Assume $d\ge x$ and call $T_d$ any degree $d$ tree containing $T$. Since $h^0(W,\Ii_{T\cap W}(t)) =0$, $h^0(W,\Ii_{T_d\cap W}(t)) =0$.

Now assume $d<x$. Assume that for a general $T\in T(n,x,\tau)$ the tree $T$ has a linking line $L_o$ which
is a final line $L_o$ with $o$ a linking point. Set $S':= W\cap T\setminus \{o\}$. Thus $h^i(W,\Ii_{S'}(t)) =0$, $i=0,1$. With
this assumption the closure
$T'$ of
$T\setminus L_o$ in
$\PP^n$ is a tree. Fix any admissible ordering $R_1,\dots ,R_{x-1}$ of the irreducible components of $T'$. For any integer
$1\le d<x$ set
$T_d:= L_1\cup \cdots L_d$. Since $T$ intersects transversally $W$, $T_d$ intersects transversally $W$. Since $T_d\cap
W\subset S'$ and $h^1(W,\Ii_{S'}(t)) =0$, $h^1(W,\Ii _{T_d\cap W}(t))=0$. If $T$ is a bamboo, then each $T_d$ is a bamboo.
\end{remark}

\begin{remark}\label{eb4}
$R(n,1)$ is false, because for instance any tree $T\subset \PP^n$ of degree $<n$ is contained in a
hyperplane and hence its intersection with $W$ is contained in a hyperplane. The same observation shows why we assumed $d\ge
n$ in Theorem \ref{eb1}.
\end{remark}

\begin{lemma}\label{rnn0}
Take a vector space $W\subseteq H^0(\Oo _{\PP^n}(2))$ and an integral curve $T\subset \PP^n$ contained in the base locus of
$W$. For any scheme
$E\subset
\PP^n$ set $W(-E):= H^0(\Ii _E(2))\cap W$. For a general
$o\in T$ and a general line $L\subset \PP^n$ such that $o\in L$ we have $\dim W(-L) = \max \{0,\dim W-2\}$, unless all $Q\in
|W|$ are cones with vertex containing $T$.
\end{lemma}

\begin{proof}
Since $W(-o) =W$, we have $\dim W(-L)\ge \dim W-2$. Since $L$ contains a general point of $\PP^n$, we have $\dim W(-L) \ge \max
\{0,\dim W-1\}$. Thus we may assume $\dim W\ge 2$. Fix a general $p\in \PP^n$. Thus $\dim W(-p) =\dim W-1$. Fix $Q\in
W(p-p)$. We are done, unless every line containing $p$ and intersecting $T$ is contained in $Q$, i.e. unless $Q$ contains the
cone $C_p(T)$ with vertex $p$ and base $T$. Take $p'\in Q$ near $p$. We still have $\dim W(-p') =\dim W(-p)$ by semicontinuity.
Thus $\dim W(-L) = \max \{0,\dim W-2\}$, unless $Q$ is a cone with vertex containing $T$. Since we may take as $Q$ any element
of $|W|$ containing a general $p\in \PP^n$, we conclude.
\end{proof}

\begin{lemma}\label{eb5}
$R(n,2)$ is true for all $n\ge 3$. 
\end{lemma}

\begin{proof}
Fix a tree $T\subset \PP^n$ transversal to $W$. Consider the exact sequence
\begin{equation}\label{eqeb1}
0\to \Ii _{T,\PP^n}\to \Ii_{T,\PP^n}(2)\to \Ii_{T\cap W,W}(2)\to 0
\end{equation}
Since $h^1(\Ii_{T,\PP^n}) =0$ and $h^2(\Ii _{T,\PP^n}(2)) =h^1(\Oo _T)=0$, the long cohomology exact sequence of \eqref{eqeb1}
gives that
$h^1(W,\Ii_{T\cap W,W}(2)) =0$ if and only if $h^1(\PP^n,\Ii_{T,\PP^n}(2))$ and a similar statement holds for $h^0$. Thus to
prove the lemma it is sufficient to prove that either $h^1(\PP^n,\Ii_{T,\PP^n}(2))=0$ or $h^0(\PP^n,\Ii_{T,\PP^n}(2))=0$. This
is easily proved by induction on the integer $\deg(T)$ using Lemma \ref{rnn0}.
\end{proof}

\begin{lemma}\label{eb6}
Fix integers $n\ge 4$ and $t\ge 3$. Assume Theorem \ref{eb1} in $\PP^{n-1}$, $R(n-1,t)$ and $R(n,t-1)$. Then $R(n,t)$ is true.
\end{lemma}

\begin{proof}
Set $x:= \lceil h^0(\Oo_W(t))/2\rceil$ and $y:= \lceil h^0(\Oo_W(t-1))/2\rceil$. Let $Y\subset \PP^n$ be a bamboo of degree $y$
satisfying $R(n,t-1)$ and general among the trees with that type. Note that  $h^0(\Oo _W(t))-h^0(\Oo_W(t-1)) = h^0(\Oo_D(t))$
for any quadric hypersurface $D$ of $\PP^{n-1}$. Thus $h^0(\Oo_D(t))$ is odd if and only if exactly one among $h^0(\Oo_W(t))$
and $h^0(\Oo_W(t-1))$ is odd.

\quad (a) Assume $h^0(\Oo _W(t)) \equiv
h^0(\Oo_W(t-1)) \equiv 0\pmod{2}$. Fix a general hyperplane $H\subset \PP^n$. Since $D:= W\cap H$ is a quadric hypersurface of
$H$ with rank $\min \{\rho,n\}$, we may apply the inductive assumption to the quadric hypersurface $W\cap H$ of $H$.
Since $h^0(\Oo _W(t))-h^0(\Oo_W(t-1)) = h^0(\Oo_D(t))$, we have $\deg (T)=x-y$. Take a final line $L$ of $Y$. For a general $Y$
we may assume that one of the point,
$o$, of
$L\cap H$ is a general point of $H$. Thus there is a solution
$T\subset D$ for $R(n-1,t)$ containing $o$ with $o$ a smooth point of $T$ contained in a final line $R$ of $T$. Deforming $T$
among the bamboos containing
$o$ we may  assume
$T\cap (Y\cap H)=\{o\}$. Thus $Y\cup T$ is a degree $x$ tree. Since $\{o\}= L\cap R$, $L$ is a final line of $Y$ and
$R$ is a final line of $T$, $Y\cup T$ is a bamboo. Consider the exact sequence of coherent sheaves on
$W$ induced by the residual exact sequence of the Cartier divisor $D$ of $W$:
\begin{equation}\label{eqbe2}
0 \to \Ii _{W\cap Y}(t-1)\to \Ii _{W\cap (Y\cup T)}(t) \to \Ii _{D\cap T,D}(t)\to 0
\end{equation}
Since $h^i(\Ii_{W\cap Y}(t-1)) =0$, $i=0,1$, and $h^i(D,\Ii _{D\cap T,D}(t))=0$, $i=0,1$, the long cohomology exact sequence
of \eqref{eqbe2} shows that $Y\cup T$ satisfies $R(n,t)$.

\quad (b) Assume $h^0(\Oo _W(t-1))$ even and $h^0(\Oo _W(t))$ odd. Thus $h^0(\Oo _D(t))$ is odd. Let $L$ be a final line of
$Y$. For a general
$Y$ we may assume that one of the point,
$o$, of
$L\cap H$ is a general point of $H$. Thus there is a solution
$T\subset D$ for $R(n-1,t)$ containing $o$ with $o$ a smooth point of $T$ contained in a final line $R$ of $T$. Deforming $T$
among the bamboos containing
$o$, we may  assume
$T\cap (Y\cap H)=\{o\}$. Thus $Y\cup T$ is a degree $x$ tree. Let $L_a$ the final line of $T$ different from $R$. Call $a$ a
linking point of
$T$ contained in $L_a$. Recall that we got the point $o\in T$ moving a solution of $R(n,t-1)$. We may move it differently in
such a way that the line containing $o$ is not $L_a$. Now $L_a$ is a final line of $Y\cup T$ and $h^i(W,\Ii _{W\cap (Y\cup
T)\setminus
\{a\})}(t))=0$, $i=0,1$, by \eqref{eqbe2} proving $R(n,t)$ in this case.

\quad ({c}) Assume $h^0(\Oo _W(t)) \equiv
h^0(\Oo_W(t-1)) \equiv 1\pmod{2}$. Call $b\in Y\cap W$ a point belong to a final line $L_b$ of $Y$ and such that
$h^i(\Ii_{W\cap Y\setminus \{b\}}(t-1)) =0$, $i=0,1$. We take the construction of step (a) with $b\notin H$ and $o\notin L_a$.
With these modifications the proof of step (a) shows that $Y\cup T$ satisfies $R(n,t)$.

\quad (d) Assume $h^0(\Oo _W(t-1))$ odd and $h^0(\Oo _W(t))$ even. Thus $h^0(\Oo_W(t)) -h^0(\Oo _W(t-1))$ is odd and $\lceil
(h^0(\Oo_W(t)) -h^0(\Oo _W(t-1)))/2\rceil +y=x+1$. Call $o$ a linking point of
$Y$ and
$L_o$ the final line of
$Y$. Let $H\subset \PP^n$ be
a general hyperplane containing $L_o$. Since
$Y$ is general, we may see $L_o$ as a general line of $H$. Since $H$ is general among the hyperplane
containing $L_o$, $D:= W\cap H$ is a quadric hypersurface of $H$ with rank $\min \{\rho,n\}$. Thus we may apply the inductive
assumption to the quadric hypersurface $D:= W\cap T$.  Thus we may find a solution $T$ of $R(n-1,t)$ with the additional
condition that $L_o$ is a linking line of $T$ with $a$ a linking point, where $\{o,a\}:= L_o\cap W$. Thus $Y\cup T$ is a
degree $x$ bamboo satisfying $R(n,t)$.
\end{proof}

\begin{remark}\label{eb7}
Let $S\subset \PP^n$ be a finite set. It is a well-known and easy exercise that $h^1(\PP_n,\Ii_{S,\PP^n}(t)) =0$ for all
$t\ge
\#S-1$.
\end{remark}

\begin{lemma}\label{eb3}
Fix an integer $n\ge 3$. If $R(n,t)$ is true for all integers $t\ge 2$, then Theorem \ref{eb1} is true in $\PP^n$.
\end{lemma}

\begin{proof}
Fix an integer $d\ge n$. Let $B(n,d)'$ denote the set of all $T\in B(n,d)$ transversal to $W$. $B(n,d)'$ is a non-empty
open subset of the irreducible quasi-projective variety $B(n,d)$. Since for any $t$ the restriction map $H^0(\Oo_{\PP^n}(t))
\to H^0(\Oo_W(t))$ is surjective, it is sufficient to prove that a general $T\in B(n,d)$ has the property that $W\cap T$
has maximal rank with respect to the line bundle $\Oo_W(t)$ for all $t>0$. 

Since $d\ge n$ a general $T\in B(n,d)'$ spans $\PP^n$, i.e.
$h^0(\PP^n,\Ii _T(1)) =0$. Thus we have the residual exact sequence
\begin{equation}\label{eqeb3}
0 \to \Ii_{T,\PP^n}(-1)\to \Ii_{T,\PP^n}(1)\to  \Ii _{T\cap W,W}(1)\to 0
\end{equation}
Since $T$ is a reduced curve, $h^0(\Oo_T(-1)) =0$. Thus \eqref{eqeb3} gives $h^0(W,\Ii _{T\cap W,W}(1))=0$. Thus $T\cap W$
has maximal rank with respect to the line bundle $\Oo_W(1)$. Thus it is sufficient to handle $W\cap T$ with respect to all
line bundles $\Oo_W(t)$, $t\ge 2$. For any $Y\in B(n,d)'$ the scheme $W\cap Y$ is a set with cardinality $2d$.
By Remark \ref{eb7} we have $h^1(W,\Ii_{W\cap Y,W}(t)) =0$ for all $t\ge 2d-1$. Thus it is sufficient to check finitely many
line bundles $\Oo_W(t)$, $2\le t\le 2d-2$. Since $B(n,d)'$ is irreducible, the intersection of finitely many non-empty open
subsets of $B(n,d)'$ is non-empty and hence open and dense in $B(n,d)'$. Thus it is sufficient to handle a single $\Oo_W(t)$
with $t\ge 2$. Set $x:= \lceil h^0(\Oo_W(t)/2\rceil$. Take a solution $X$ of $R(n,d)$ and fix an admissible ordering of the
irreducible components $L_1,\dots ,L_x$ of $X$. If $h^0(\Oo_W(t))$ is odd assume that $L_d$ is a linking line and fix a
linking point $o\in L_d\cap W$.

First assume $h^0(\Oo_W(t))$ even. If $d =x$ we may take $X$ as $T$, because $h^i(W,\Ii_{X\cap W,W}(t))=0$, $i=0,1$. If $d>x$
we may take as $T$ any element of $B(n,d)'$ containing $X$, because $h^0(W,\Ii_{X\cap W,W}(t))=0$ implies $h^0(W,\Ii_{T\cap
W,W}(t))=0$. If $d<x$ we may take as $T$ any degree $d$ bamboo contained in $X$, because $h^1(W,\Ii_{X\cap W,W}(t))=0$ implies
$h^1(W,\Ii_{T\cap W,W}(t))=0$.

Now assume $h^0(\Oo_W(t))$ odd. If $d =x$ we may take $X$ as $T$, because $h^i(W,\Ii_{X\cap W,W}(t))=0$, $i=0,1$. If $d>x$
we may take as $T$ any element of $B(n,d)'$ containing $X$, because $h^0(W,\Ii_{X\cap W,W}(t))=0$ implies $h^0(W,\Ii_{T\cap
W,W}(t))=0$. Now assume $d<x$. Set $E:= L_1\cup \cdots \cup L_{d-1}$. We have $h^1(W,\Ii_{E\cap W,W}(t))=0$ by $R(n,t)$. Hence
we may take as $T$ any degree $d$ bamboo contained in $E$.
\end{proof}

\begin{proof}[Proof of Theorem \ref{eb1}:]
Since the case $n=3$ is true by Proposition \ref{be1}, we may assume $n\ge 4$. Use Lemmas \ref{eb3} and \ref{eb7}.
\end{proof}

\begin{proof}[Proof of Proposition \ref{fe1}:]
By the semicontinuity theorem for cohomology  is sufficient to prove the proposition for a very specific degree $k$ hypersurface: a hyperplane $H$ counted with
multiplicity $k$. Call $M_e$, $1\le e\le k$, the hyperplane counted with multiplicity $e$. For a general $Y$ the set $S:=
Y\cap H$ is a general subset of $H$ with cardinality $d$. Thus $h^0(H,\Ii_{S,H}(x))=\max \{\binom{n+x-1}{n-1}-d,0\}$
and  $h^1(H,\Ii_{S,H}(x))=\max \{d-\binom{n+x-1}{n-1},0\}$. Then we use the exact sequences
$$0\to \Oo_H(-e-1)\to\Oo_{M_{e+1}}\to\Oo_{M_e}\to 0$$and induction on the integer $e$.\end{proof}

\section{Proof of Theorem \ref{be2}}

In this section we work in $\PP^3$ and prove Theorem \ref{be2}. Let $W\subset \PP^3$ be a smooth cubic surface. Take a
general smooth quadric $Q\subset \PP^3$. Bertini's theorem gives that $D:= Q\cap W$ is a
smooth element of $|\Oo_Q(3,3)|$. Since $W$ is general, $D$ is a general element of $|\Oo_Q(3,3)|$. Thus $D$ is a general
smooth curve of genus $4$ canonically embedded in $\PP^3$.

Since $h^1(\Oo_{\PP^3}(t-3)) =0$ for all $t\in \ZZ$ and $h^2(\Oo_{\PP^n}(t-3)) =0$ for all $t\ge 0$, a standard exact sequence
gives
$h^0(\Oo_W(t)) = \binom{t+3}{3} -\binom{t}{3}$ and $h^1(\Oo_W(t-2))=0$ for all $t\ge 1$. We also have $h^0(\Oo_D(t)) =
6t-3$ for all $t\ge 4$. Thus $h^0(\Oo_W(t)) -h^0(\Oo_W(t-2)) = h^0(\Oo _D(t))$ for all $t\ge 2$.
Set $x_t:= \lfloor h^0(\Oo_W(t))/3\rfloor$. Note that $h^0(\Oo_W(t)) \equiv 1 \pmod{3}$ for all $t\ge 0$, i.e.
$h^0(\Oo_W(t))=3x_t +1$ for all $t\ge 0$. Since $h^0(\Oo_W(t)) -h^0(\Oo_W(t-2)) = 6t-2$ for all
$t\ge 2$, we have $x_t =x_{t-2}+2t-1$ for all $t\ge 4$.  For any positive integer $d$ let $T(3,d)'$ denote the set of all
$Y\in T(3,d)$ transversal to
$W$. For any type $\tau$ for degree $d$ trees set $T(3,d,\tau)':= T(3,d)'\cap T(3,d,\tau)$. Call $B(3,d)$ the set of all
degree $d$ bamboos and set $B(3,d)':= B(3,d)\cap T(3,d)'$.

\begin{remark}\label{eeb1}
For any $Y\in T(3,d)'$
the scheme $Y\cap W$ is the union of $3d$ points. Thus $Y\cap W$ has maximal rank in $W$ if and only if $h^0(W,\Ii_{W\cap
Y}(t))=0$
for all $t$ such that $x_t<d$ and $h^1(W,\Ii_{W\cap Y}(t))=0$ if $d\le x_t$.
\end{remark}
Consider the following assertion $H(t)$, $t\ge 3$:

\quad {\bf Assertion} $H(t)$, $t\ge 3$: There is a type $\tau$ for degree $x_t$ trees such that $h^1(W,\Ii _{W\cap Y}(t)) =0$
for a general
$Y\in T(3,x_t,\tau)'$.

Note that Assertion $H(t)$ is true if and only if $h^0(W,\Ii _{W\cap Y}(t)) =1$ for a general $Y\in T(3,x_t,\tau)'$. By the
semicontinuity theorem for cohomology $H(t)$ is true if and only if there is at least one $Y\in T(3,x_t,\tau)'$.

\begin{remark}\label{be3}
Let $T\subset \PP^3$ be any degree $3$ tree. Since $\deg (T)\le 3$, $T$ is a bamboo. Since $\deg (T)=3$ and $p_a(T)=0$, $T$
spans $\PP^3$. Thus $h^1(\Ii_T(1)) =0$. Since $h^2(\Ii_T) =h^1(\Oo_T)=0$, the Castelnuovo-Mumford's lemma gives $h^1(\Ii_T(t))
=0$ for all $t\ge 5$.
\end{remark}

\begin{lemma}\label{be4}
Fix a type $\tau$ for degree $4$ trees. Let $T\subset \PP^3$ be a general degree $4$ tree with type $\tau$. Then
$h^1(\Ii_T(t))=0$ for all $t=2$.
\end{lemma}

\begin{proof}
By the Castelnuovo-Mumford's lemma it is sufficient to prove the case $t=2$.

First assume that $T$ is not a bamboo, i.e. assume
$\tau(i)=1$ for all
$i=2,3,4$. By semicontinuity it is sufficient to find one tree
$Y$ with type $\tau$ and $h^1(\Ii_Y(2)) =0$, i.e. with $h^0(\Ii_Y(2))=0$. Fix a smooth quadric $Q\subset \PP^3$. Fix $L_1\in
|\Oo Q(1,0)|$ and $3$ distinct elements $L_i\in |\Oo_Y(0,1)|$, $i=2,3,4$. Set $T:=L_1\cup L_2\cup L_3\cup L_4\in
|\Oo_Y(1,3)|$. $T$ is a tree of type $\tau$. Since $h^1(\Oo_{\PP^3})=0$ and $h^0(\Oo_Q(-1,1))=0$, $Q$ is the only quadric
containing $T$.

Now assume that $\tau$ is a bamboo. Take two different planes $H$ and $M$. Fix general lines $L_1, L_2$ of $H$ and a general
line $L_4$ of $M$. Let $L_3\subset M$ a general line containing the point $L_3\cap M$. Note that $T:= L_1\cup L_2\cup L_3\cup
L_4$ is a bamboo contained in $H\cup M$. Take any $W\in |\Ii_T(2)|$. Since $W_{|M}$ contains $L_3\cup L_4\cup \{L_1\cap M\}$,
$M$ is an irreducible component of $W$. Similarly, $H\subset W$. Thus $|\Ii_T(2)|=\{H\cup M\}$.
\end{proof}

\begin{lemma}\label{be4.1}
Fix a general $Y\in B(3,5)'$. We have $h^0(W,\Ii_{W\cap Y,W}(x)) =0$ for all $x\le 2$ and $h^1(W,\Ii_{W\cap Y,W}(t))=0$ for
all
$t\ge 3$. 
\end{lemma}

\begin{proof}
Since $W\cap Y$ spans $\PP^3$, $h^0(W,\Ii_{W,\PP^3}(1)) =0$. Fix a general $E\in B(3,4)$. Lemma \ref{be4} gives
$h^0(\PP^3,\Ii _{E,\PP^3}(2))=1$. Adding a line to $E$ we get a degree $5$ bamboo $Y$ such that $h^0(\PP^3,\Ii _{Y,\PP^3}(2))=0$. Since $h^1(\PP^3,\Ii
_{Y,\PP^3})=0$, the residual exact sequence of $W$ gives $h^0(W,\Ii_{W\cap Y,W}(2)) =0$. 

\quad {\bf Claim:} $h^1(\PP^3,\Ii_{Y,\PP^3}(3))=0$.

\quad {\bf Proof of the Claim:} Let $Q\subset \PP^3$ be a smooth quadric. Take $3$ distinct elements $L_1,L_3,L_5$
of $|\Oo_Q(1,0)|$. Let $L_2$ be a general line of $\PP^3$ intersecting both $L_1$ and $L_3$. Let $L_4$  be a general line of
$\PP^3$ intersecting both $L_5$ and $L_3$. Set $F:= L_1\cup L_2\cup L_3\cup L_4\cup L_5$. Since $F\in B(3,5)$, by
semicontinuity  to prove the claim it is sufficient to prove that $h^1(\PP^3,\Ii_{F,\PP^3}(3))=0$. Note that $F\cap Q=L_1\cup
L_3\cup L_5$ scheme-theoretically. Thus the residual exact sequence of $Q$ gives the following exact sequence
\begin{equation}\label{eqbe4.1}
0 \to \Ii_{L_2\cup L_4,\PP^3}(1) \to \Ii _{F,\PP^3}(3)\to \Ii_{L_1\cup L_3\cup L_5,Q}(3)\to 0
\end{equation}
Use $h^1(\PP^3,\Ii _{L_2\cup L_4,\PP^3}(1))=0$, $h^1(Q,\Ii_{L_1\cup L_3\cup L_5,Q}(3))=h^1(Q,\Oo_Q(0,3)) =0$ and the
cohomology exact sequence of \eqref{eqbe4.1}.

By the Castelnuovo-Mumford's lemma to prove the $h^1$-vanishing
it is sufficient to prove that $h^1(W,\Ii_{Y\cap W}(3)) =0$, which is true by the residual exact sequence of $W$, because
$h^1(\PP^3,\Ii_{Y,\PP^3})=h^1(\PP^3,\Ii_{Y,\PP^3}(3))=0$.
\end{proof}

\begin{lemma}\label{be5}
There is a degree $6$ bamboo $T\subset \PP^3$ such that $h^1(\Ii_T(3)) =0$.
\end{lemma}

\begin{proof}
Fix a plane $H\subset \PP^3$ and a general reducible conic $L_5\cup L_6\subset H$.  Fix a general $o\in L_4$.
Let $Y:= L_1\cup L_2\cup L_3\cup L_4$ be a general degree $3$ bamboo containing $o$. Thus $T:= Y\cup L_5\cup L_6$ is a degree
$6$
line bundle. Consider the residual exact sequence of $H$:
\begin{equation}\label{eqbe1}
0\to \Ii_Y(2)\to \Ii_T(3) \to \Ii_{T,H}(3)\to 0
\end{equation}
Since $T\cap H$ is the union of $L_5\cup L_6$ and $3$ general points of $H$, we have $h^0(H,\Ii_{T\cap H,H}(3)=0$,i.e.
$h^1(H,\Ii_{T\cap H,H}(3))=0$. Lemma \ref{be4} gives $h^1(\Ii_Y(2))=0$. Use the long cohomology exact sequence of
\eqref{eqbe1}.
\end{proof}

\begin{lemma}\label{be6}
Let $W\subset \PP^3$ be any degree $3$ surface (even reducible or with multiple components). Let $Y\subset \PP^3$ be a general
degree $6$ bamboo and $X\subset \PP^3$ be a general degree $7$ bamboo. The $\dim Y\cap W = \dim X\cap W =0$, $h^1(W,\Ii_{Y\cap
W,W}(3)) =0$ and $h^0(W,\Ii_{W\cap X,W}(3))=0$.
\end{lemma}

\begin{proof}
General $Y$ and $X$ have no irreducible components containing in $W$. Thus the residual exact sequence of $W$ gives the exact
sequence
\begin{equation}\label{eqbe2}
0\to \Ii _Y\to \Ii_Y(3)\to \Ii_{W\cap Y,W}(3)\to 0
\end{equation}
Lemma \ref{be5} gives $h^1(\Ii_Y(3))=0$. Since $h^2(\Ii _Y)=h^1(\Oo _Y)=0$, the long cohomology exact sequence of \eqref{eqbe2}
gives
$h^1(W,\Ii_{W\cap Y,Y}(3))=0$, i.e. $h^0(W,\Ii_{W\cap Y,Y}(3))=1$, say $\{D\}:= |\Ii_{W\cap Y,W}(3)|$. Take as $X$ the union
of $Y$ and a general line $L$ intersecting a final line of $T$. Observe that $L\cap W\nsubseteq D$.
\end{proof}

\begin{remark}\label{eeb2}
Fix a positive integer $d$, a type $\tau$ for degree $d$ trees and a general $Y\in T(3,d,\tau)'$,  and set $S:= Y\cap W$. Thus
$\#S =3d$.

\quad (a) Assume $d=1$. $S$ is formed by $3$ collinear points. Thus $h^1(\PP^3,\Ii_{S,\PP^3}(1)) =1$, $h^0(\PP^3,\Ii
_{S,\PP^3}(1))=2$ and $h^1(\PP^3,\Ii_{S,\PP^3}(t)) =0$ for all $t\ge 2$. Thus $h^1(W,\Ii_{S,\PP^3}(1)) =1$, $h^0(W,\Ii
_{S,W}(1))=2$ and $h^1(W,\Ii_{S,W}(t)) =0$ for all $t\ge 2$.

\quad (b) Assume $d=2$. Thus $S$ is the union of $6$ coplanar points contained in a reducible plane conic and in no other
conic. Thus $h^1(\PP^3,\Ii_{S,\PP^3}(1)) =3$, $h^0(\PP^3,\Ii
_{S,\PP^3}(1))=1$, $h^1(\PP^3,\Ii_{S,\PP^3}(2))=1$, $h^0(\PP^3,\Ii _{S,\PP^3}(2)) =5$  and $h^1(\PP^3,\Ii_{S,\PP^3}(t)) =0$
for all
$t\ge 3$. Thus $h^1(W,\Ii_{S,\PP^3}(1)) =1$, $h^0(W,\Ii _{S,W}(1))=2$, $h^1(W,\Ii_{S,W}(2))=1$, $h^0(W,\Ii
_{S,W}(2)) =5$ and $h^1(W,\Ii_{S,W}(t)) =0$ for all $t\ge 3$.

\quad ({c}) Assume $d=3$. Obviously $h^0(\PP^3,\Ii _{Y,\PP^3}(1))=0$. Since $h^2(\PP^3,\Ii_{Y,\PP^3}))=h^1(\Oo_Y)=0$, the
Castelnuovo-Mumford's lemma gives $h^1(\PP^3,\Ii _{Y,\PP^3}(t))=0$ for all $t\ge 2$. Hence $h^0(\PP^3,\Ii_Y(2))=3$. Obviously
$h^0(W,\Ii_{W\cap Y,W}(2)) \ge 3$. Since $h^1(\PP^3,\Ii_Y(-1))=0$, the residual exact sequence of $W$ gives $h^0(W,\Ii_{W\cap
Y,W}(2)) =3$. Thus $h^1(W,\Ii_{W\cap Y,W}(2)) =2$. Since $h^1(\PP^3,\Ii_{Y,\PP^3})=0$, the residual exact sequence of $W$
gives $h^0(W,\Ii_{Y\cap W,W}(3)) =h^0(\PP^3,\Ii_{Y,\PP^3}(3)) =10$. Thus $h^1(W,\Ii_{Y\cap W,W}(3)) =0$. The
Castelnuovo-Mumford's lemma gives $h^1(W,\Ii_{Y\cap W,W}(3)) =0$.

\quad (d) Assume $d=4$. Lemma \ref{be4} gives $h^0(\PP^3,\Ii_{Y,\PP^3}(2))=1$. Since $h^1(\PP^3,\Ii_{Y,\PP^3}(-1))=0$, the
residual exact sequence of $W$ gives $h^0(W,\Ii_{Y\cap W,W}(2))=1$. Thus we have $h^1(W,\Ii_{Y\cap W,W}(2))=3$.

\end{remark}

\begin{lemma}\label{eeb3}
$H(3)$ is true.
\end{lemma}

\begin{proof}
Note that $x_3=6$. Apply Lemma \ref{be4}.
\end{proof}

\begin{lemma}\label{eeb4}
Fix an integer $t\ge 5$ and assume $H(t-2)$. Then $H(t)$ is true.
\end{lemma}

\begin{proof}
Fix a solution $Y\in T(3,x_{t-2},\tau)'$. Moving $Y$ we may assume that $Y$ is transversal to $Q$, that no line of $Q$ is
secant to $Y$ and that $Q\cap D=\emptyset$. We fix
$o\in Y\cap Q$ and take $E\in |\Oo_Q(1,2t-2)|$ formed by $2t-1$ distinct lines with the component in $|\Oo_Q(1,0)|$ containing
$o$ and all other ones not intersecting $Y\cap Q$. Thus $Y\cup E$ is a degree $x_t$ tree. The residual exact sequence of $D$ in
$W$ shows that to prove $H(t)$ it is sufficient to prove that $h^0(D,\Ii _{E\cap D,D}(t,t))=0$. The line bundles $\Oo_D(1,0)$ and
$\Oo _D(0,1)$ are the two $g^1_3$'s on the general genus $4$ curve. Thus it is sufficient to observe that
$h^0(D,\Oo_D(t-1,-t+2))=0$ by the generality of $D$ and the strong Franchetta conjecture (\cite[Theorem 2]{kou}, \cite{mes}).
\end{proof}

\begin{lemma}\label{efb1}
$H(4)$ is true.
\end{lemma}

\begin{proof}
We have $x_4 = 10$. Let $A\subset \PP^3$ be a general union of $3$ lines. We have $h^0(\PP^3,\Ii
_{E,\PP^3}(2)) =1$ (easy, take $E\subset Q$ or apply \cite{hh}) and hence $h^1(\PP^3,\Ii_{E,\PP^3}(2))=0$. Since
$h^1(\Oo_{\PP^1}(-1)) =0$. Thus $h^2(\PP^3,\Ii_{E,\PP^3}(-1)) =h^1(\Oo_E(-1))=0$. The residual exact sequence of $W$ gives
$h^1(W,\Ii_{E\cap W,W}(2))=0$. We take $F\in |\Oo_Q(1,6)|$ union of $7$ lines, $3$ of the one of the ruling $|\Oo_Q(0,1)|$
meeting a different line of $E$. Thus $E\cup F$ is a tree. Then we continue as in the proof of Lemma \ref{eeb4}.
\end{proof}

\begin{proof}[Proof of Theorem \ref{be2}:]
We discussed in Remark \ref{eeb2} why the listed case are exceptional and the amount of maximal rank failure for each of these
cases.
Thus to prove the theorem we may assume $d\ge 5$. For $d=5$ use Lemma \ref{be4.1}. For $d=6$ use Lemma \ref{be6}. 
From now on we assume $d\ge 7$.

By Remark \ref{eeb2} it is sufficient to prove that for all $d\ge 6$ the set $W\cap Y$ has maximal rank in $W$ for some $\tau$
and a general
$Y\in T(3,d,\tau)'$. Let $t$ be the minimal integer such that $x_t\ge d$. Thus $x_{t-1}<d\le x_t$. Since $H(t)$ is true, there
is a type
$\gamma$ for degree $x_t$ trees such that $h^0(W,\Ii_{T\cap W}(t)) =1$ and $h^1(W,\Ii _{T\cap W}(t))=0$ for a general $T\in
T(3,x_t,\gamma)$. Thus
$h^1(W,\Ii _{T\cap W}(x))=0$ for all
$x\ge t$ by the Castelnuovo-Mumford's lemma. Thus $h^1(W,\Ii_{W\cap E,W}(x)) =0$ for all $x\ge t$ and all curves $E\subseteq
T$. To conclude it is sufficient to prove the existence of a degree $d$ connected curve $E\subseteq T$ such that $h^0(W,\Ii
_{W\cap E,W}(t-1)) =0$. We need the proof of Lemma \ref{eeb4}. We obtained $Y$ starting from a solution $Y'$ of $H(t-2)$
and adding $2t-1$ lines contained in $Q$. Adding $2t-1-x_t+d$ lines we get  $h^0(W,\Ii
_{W\cap (Y'\cup E'),W}(t-1)) =0$, because $d>x_{t-1}$.
\end{proof}

\providecommand{\bysame}{\leavevmode\hbox to3em{\hrulefill}\thinspace}

\end{document}